\documentclass[a4paper]{amsart}
\usepackage{graphicx} 

\usepackage[utf8]{inputenc}
\usepackage[a-2u]{pdfx}
\usepackage{lmodern, booktabs}
\usepackage{amsmath,amsthm,amssymb}
\usepackage{pgfplots}
\usepgfplotslibrary{fillbetween} 
\usepackage{float}
\usepackage{tikz}

\usetikzlibrary{matrix}
\usetikzlibrary{patterns, shapes}
\usetikzlibrary{arrows}
\usetikzlibrary{calc,3d}
\usetikzlibrary{decorations,decorations.pathmorphing, decorations.pathreplacing}
\usetikzlibrary{through}
\tikzset{ext/.style={circle, draw,inner sep=1pt},int/.style={circle,draw,fill,inner sep=1pt},nil/.style={inner sep=1pt}}
\tikzset{exte/.style={circle, draw,inner sep=3pt},inte/.style={circle,draw,fill,inner sep=3pt}}
\tikzset{diagram/.style={matrix of math nodes, row sep=3em, column sep=2.5em, text height=1.5ex, text depth=0.25ex}}
\tikzset{diagram2/.style={matrix of math nodes, row sep=0.5em, column sep=0.5em, text height=1.5ex, text depth=0.25ex}}
\tikzset{every picture/.style={baseline=-.65ex}}
\tikzstyle{every loop}=[draw]
\tikzstyle{rloop}=[ out=10, in=-10, loop, distance=3em] 
\tikzstyle{aloop}=[ out=100, in=80, loop, distance=3em] 
\usepackage{tikz-cd}

\newcommand{\Out}{\mathrm{Out}}
\newcommand{\Aut}{\mathrm{Aut}}
\DeclareMathOperator{\gr}{gr}
\DeclareMathOperator{\rank}{rank}

\newcommand{\Z}{\mathbb{Z}}
\newcommand\FG{\mathsf{FG}}
\newcommand{\FGO}{\widetilde{\FG}}
\newtheorem{theorem}{Theorem}
\newtheorem{proposition}[theorem]{Proposition}
\newtheorem{lemma}[theorem]{Lemma}
\newtheorem{corollary}[theorem]{Corollary}

\newtheorem{conjecture}[theorem]{Conjecture}

\theoremstyle{definition}

\newtheorem{example}[theorem]{Example}

\newcommand{\sgn}{\mathrm{sgn}}
\theoremstyle{definition}

\theoremstyle{remark}

\DeclareMathOperator{\vdim}{dim}

\title{The homology of $\Out(F_8)$ and $\Aut(F_7)$}

\author{Maro\v s Grego}
\address{Department of Algebra, Faculty of Mathematics and Physics, Charles University, So\-ko\-lovsk\'a 83, 186~00 Prague 8, Czech Republic and Institute of Mathematics, Czech Academy of Sciences, \v Zitn\'a 25, 115 67 Prague 1, Czech Republic}
\email{grego@math.cas.cz}

\author{Tom\'a\v s V\'itek}
\address{Mathematical Institute, Faculty of Mathematics and Physics, Charles University, So\-ko\-lovsk\'a 83, 186~00 Prague 8, Czech Republic and Institute of Mathematics, Czech Academy of Sciences, \v Zitn\'a 25, 115 67 Prague 1, Czech Republic}
\email{tomas.vitek.university@gmail.com} 

\author{Thomas Willwacher}
\address{Department of Mathematics\\ ETH Zurich\\ R\"amistrasse 101 \\ 8092 Zurich, Switzerland}
\email{thomas.willwacher@math.ethz.ch}

\date{}

\begin{document}

\begin{abstract}
    We compute the rational homology of $\Out(F_8)$ and $\Aut(F_7)$, and several related groups. We also show that the Eisenstein classes in $H_7(\Aut(F_5))$ and $H_{11}(\Aut(F_7))$ are zero.
\end{abstract}

\maketitle

\section{Introduction}

Let $F_n$ be the free group with $n$ generators and let $\Aut(F_n)$ be the group of its automorphisms.
Let $\mathrm{Inn}(F_n)\subset \Aut(F_n)$ be the subgroup of \textit{inner} automorphisms, i.e., the automorphisms of the form
$x \mapsto g^{-1}xg$ for some $g \in F_n$.
The group of \textit{outer automorphisms} of $F_n$ is then the quotient group
$\Out(F_n) := \Aut(F_n) / \mathrm{Inn}(F_n)$.

The groups $\Out(F_n)$ and $\Aut(F_n)$ and their homology have received significant interest in geometric group theory.
In the following, homology is always considered with rational coefficients. 
Culler and Vogtmann have shown that the virtual cohomological dimension of $\Aut(F_n)$ is $2n-2$ and that of $\Out(F_n)$ is $2n-3$, see \cite[Corollary 6.1.3]{CullerVogtmann86}, \cite{VogtmannSurvey}. Furthermore, it is known that \cite{ConantHatcherKassabovVogtmann, HatcherVogtmann1, HatcherVogtmann2, Galatius}
\begin{align*}
    H_k(\Aut(F_n)) &=H_k(\Out(F_n))= 0 &\text{for $1\leq k\leq \frac 45 n -1$} 
\end{align*}
Possibly the most well-known family of homology classes are the Morita classes $\mu_{k}\in H_{4k}(\Out(F_{2k+2}))$. They have been shown to be nontrivial very recently by Kupers, Miller and Patzt \cite{KMP}. 
A large set of further possibly nontrivial cycles has been proposed in \cite{ConantHatcherKassabovVogtmann}.
Furthermore, extensive tables of Euler characteristics have been computed in \cite{BorinskyVermaseren}.

One may likewise study the homology with coefficients in the sign (or determinant) representation $H_k(\Aut(F_n),\mathrm{sgn})$ and $H_k(\Out(F_n),\mathrm{sgn})$.
These also contain a series of homology classes analogous to the Morita classes $\nu_{2k}^{\mathrm{BW}},\nu_{2k}^{\mathrm{KMP}}\in H_{2k-1}(\Out(F_{2k}),\sgn)$, introduced in \cite{BrunWillwacher} and \cite{KMP}, respectively. 
The classes $\nu_{2k}^{\mathrm{KMP}}$ are also shown to be non-trivial in \cite{KMP} for $k\geq 2$. The same argument used there also implies that the classes $\nu_{2k}^{\mathrm{BW}}$ are non-trivial, but the details are left to a forthcoming work. Conjecturally, $\nu_{2k}^{\mathrm{BW}}$ and $\nu_{2k}^{\mathrm{KMP}}$ are equal up to a multiplicative constant.

Numerically, the homology of $\Out(F_n)$ with trivial coefficients has been computed up to $n=6$ by Ohashi \cite{Ohashi}, and for $n=7$ by Bartholdi \cite{Bartholdi}.
For coefficients in the sign representation the homology of $\Out(F_n)$ was computed for $n\leq 7$ in \cite{BrunWillwacher}.
The same paper also contains the analogous computations for $\Aut(F_n)$ up to $n=6$.
Our main result is the computation of the homology of $\Out(F_8)$ and $\Aut(F_7)$, and the computation of the top homology of $\Out(F_9)$. The results can be found in Table \ref{tbl:results} and are displayed graphically in Figures \ref{fig:results} and \ref{fig:results2}. 

\begin{table}
\begin{center}
$\vdim\, H_k(\Out(F_n))$

\begin{tabular}{c|cccccccccccccccc}
n,k & 0 & 1 & 2 & 3 & 4 & 5 & 6 & 7 & 8 & 9 & 10 & 11 & 12 & 13 & 14 & 15\\
\hline
2 & 1 & 0 &  &  &  &  &  & &  &  &  &  &  \\
3 & 1 & 0 & 0 & 0 &  &  &  & &  &  &  &  &  \\
4 & 1 & 0 & 0 & 0 & 1 & 0 &  & &  &  &  &  &  \\
5 & 1 & 0 & 0 & 0 & 0 & 0 & 0 & 0 &  &  &  &  &  \\
6 & 1 & 0 & 0 & 0 & 0 & 0 & 0 & 0 & 1 & 0 &  &  &  \\
7 & 1 & 0 & 0 & 0 & 0 & 0 & 0 & 0 & 1 & 0 & 0 & 1 &  \\
8 & 1 & 0 & 0 & 0 & 0 & 0 & 0 & 0 & 0 & 0 & 0 & 0 & 3 & 3 \\
9 & 1 & 0 & 0 & 0 & 0 & 0 & 0 & ?&  ?& ? &  ?&  ?& ?& ?&  ?& 24 \\
\end{tabular}

\medskip 

$\vdim\, H_k(\Aut(F_n))$

\begin{tabular}{c|ccccccccccccc}
n,k & 0 & 1 & 2 & 3 & 4 & 5 & 6 & 7 & 8 & 9 & 10 & 11 & 12 \\
\hline
2 & 1 & 0 & 0 &  &  &  &  & &  &  &  &  &  \\
3 & 1 & 0 & 0 & 0 & 0 &  &  & &  &  &  &  &  \\
4 & 1 & 0 & 0 & 0 & 1 & 0 & 0 & &  &  &  &  &  \\
5 & 1 & 0 & 0 & 0 & 0 & 0 & 0 & 1 & 0 &  &  &  &  \\
6 & 1 & 0 & 0 & 0 & 0 & 0 & 0 & 0 & 1 & 0 & 1 &  &  \\
7 & 1 & 0 & 0 & 0 & 0 & 0 & 0 & 0 & 1 & 0 & 0 & 1 & 2  
\end{tabular}

\medskip 

$\vdim\, H_k(\Out(F_n),\sgn)$

\begin{tabular}{c|cccccccccccccccc}
n,k & 0 & 1 & 2 & 3 & 4 & 5 & 6 & 7 & 8 & 9 & 10 & 11 & 12 & 13 & 14 & 15\\
\hline
2 & 0 & 0 &  &  &  &  &  & &  &  &  &  &  \\
3 & 0 & 0 & 0 & 0 &  &  &  & &  &  &  &  &  \\
4 & 0 & 0 & 0 & 1 & 0 & 0 &  & &  &  &  &  &  \\
5 & 0 & 0 & 0 & 0 & 0 & 0 & 0 & 0 &  &  &  &  &  \\
6 & 0 & 0 & 0 & 0 & 0 & 1 & 0 & 0 & 0 & 0 &  &  &  \\
7 & 0 & 0 & 0 & 0 & 0 & 0 & 0 & 0 & 0 & 0 & 0 & 2 &  \\
8 & 0 & 0 & 0 & 0 & 0 & 0 & 0 & 1 & 0 & 0 & 0 & 0 & 0 & 7 \\
\end{tabular}

\medskip 

$\vdim\, H_k(\Aut(F_n),\sgn)$

\begin{tabular}{c|ccccccccccccc}
n,k & 0 & 1 & 2 & 3 & 4 & 5 & 6 & 7 & 8 & 9 & 10 & 11 & 12 \\
\hline
2 & 0 & 0 & 0 &  &  &  &  & &  &  &  &  &  \\
3 & 0 & 0 & 0 & 1 & 0 &  &  & &  &  &  &  &  \\
4 & 0 & 0 & 0 & 1 & 0 & 0 & 0 & &  &  &  &  &  \\
5 & 0 & 0 & 0 & 0 & 0 & 1 & 0 & 0 & 0 &  &  &  &  \\
6 & 0 & 0 & 0 & 0 & 0 & 1 & 0 & 0 & 0 & 0 & 2 &  &  \\
7 & 0 & 0 & 0 & 0 & 0 & 0 & 0 & 1 & 0 & 0 & 0 & 3 & 7  
\end{tabular}

\end{center}
    \caption{Dimension of the homology of $\Out(F_n)$ and $\Aut(F_n)$. Our new results are the rows $n=8,9$ for $\Out(F_n)$, and $n=7$ for $\Aut(F_n)$, with the other rows reproduced from \cite{Ohashi}, \cite{Bartholdi}, \cite{BrunWillwacher} for the reader's convenience.
    For $\Out(F_9)$ only the top degree homology is newly computed, the displayed low degree homology is in the stable range.
    \label{tbl:results} }
\end{table}

The tables are produced by numerically computing the homology of the forested graph complex $\FG$, or respectively the Lie graph complex, whose definition we recall in Section \ref{sec:FG} below. Implementation details and results, including further new results on the homology of the groups $\Gamma_{n,h}$ of \cite{ConantKassabovVogtmann, ConantHatcherKassabovVogtmann}, can be found in Section \ref{sec:methods results}.


We end the paper by describing in the appendix two interesting observations we noted in our experiments, which might help to simplify future exploration. 
First, the forested graph complex contains a smaller subcomplex spanned by 3-edge connected graphs. We show that the inclusion is a quasi-isomorphism (for the hairless part) in loop orders $\leq 8$, and for higher loop orders there is only a "small" discrepancy in the homology. This observation is essentially the reason for the feasibility of the computation of $H^{15}(\Out(F_9))$.
Second, we study on $\FG$ the filtration by girth, that is, by the length of the smallest cycle of the graph with the marked forest contracted. The resulting spectral sequence may be used to compute $H(\FG)$, and leads to fairly small computational problems. Experimentally, it turns out that almost all homology is concentrated in girth 1.

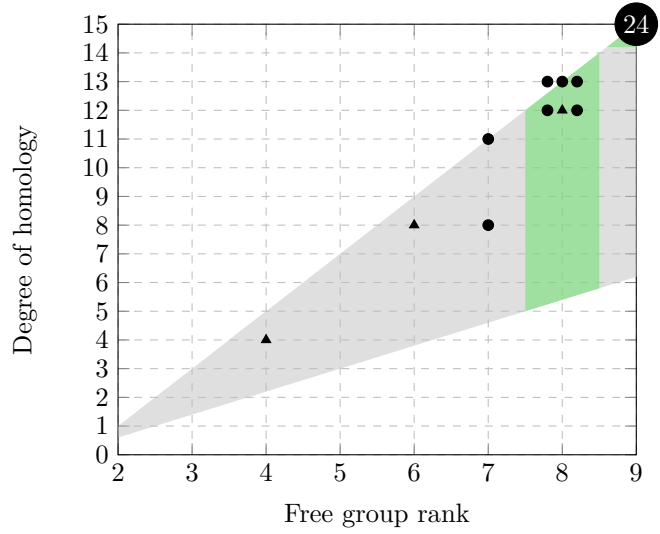
\begin{figure}
\begin{center}
\begin{tikzpicture}
\begin{axis}[
    xlabel={Free group rank},
    ylabel={Degree of homology},
    xmin=2, xmax=9,
    ymin=0, ymax=15,
	xtick distance=1,
	ytick distance=1,
    legend pos=north west,
    ymajorgrids=true,
    xmajorgrids=true,
    grid style=dashed,
	grid=both,
]

\addplot[
	only marks,
	black,
    color=black,
	fill=black,
	mark=triangle*,
    ]
    coordinates {
    (4,4)
    (6,8)
    (8,12)
    };

\addplot[
	only marks,
	black,
    color=black,
	fill=black,
	mark=*,
    ]
    coordinates {
	(7, 8)
	(7, 11)
    };

\addplot[
	only marks,
	black,
    color=black,
	fill=black,
	mark=*,
    ]
    coordinates {
	(8, 13)
	(7.8, 13)
	(8.2, 13)

	(7.8, 12)
	(8.2, 12)
    };

\addplot[
	only marks,
	black,
    color=black,
	fill=black,
	mark=*,
	mark size=8pt,
    ]
    coordinates {
	(9, 15)
    };
	
\addplot[
	only marks,
	black,
    color=white,
	fill=white,
	mark=text,
	text mark=24
    ]
    coordinates {
	(9, 15)
    };

\addplot[
	name path=upknown,
	domain=7.5:8.5,
	opacity=0,
]
	{2*x - 3};
\addplot[
	name path=upunknown,
	domain=0:12,
	opacity=0,
]
	{2*x - 3};
	
\addplot[
	name path=downknown,
	domain=7.5:8.5,
	opacity=0,
]
	{4*x/5 - 1};
\addplot[
	name path=downunknown,
	domain=0:12,
	opacity=0,
]
	{4*x/5 - 1};
\addplot[
	name path=upknown2,
	domain=8.5:9.5,
	opacity=0,
]
	{2*x - 3};
\addplot[
	name path=downknown2,
	domain=8.6:9.5,
	opacity=0,
]
	{14.2};

\addplot [green, opacity=0.5] fill between [of=upknown and downknown];
\addplot [green, opacity=0.5] fill between [of=upknown2 and downknown2];
\addplot [lightgray, opacity=0.5] fill between [of=upunknown and downunknown];
    
\end{axis}
\end{tikzpicture}
\end{center}

\caption{\label{fig:results} The known rational homology of $\Out(F_n)$. The shaded part is the potentially non-trivial range,
the green part contains the new results. The triangles are the Morita classes, the dots
are the other non-trivial classes. The top right dot contains the dimension
of the top class in $\Out(F_9)$.}
\end{figure}

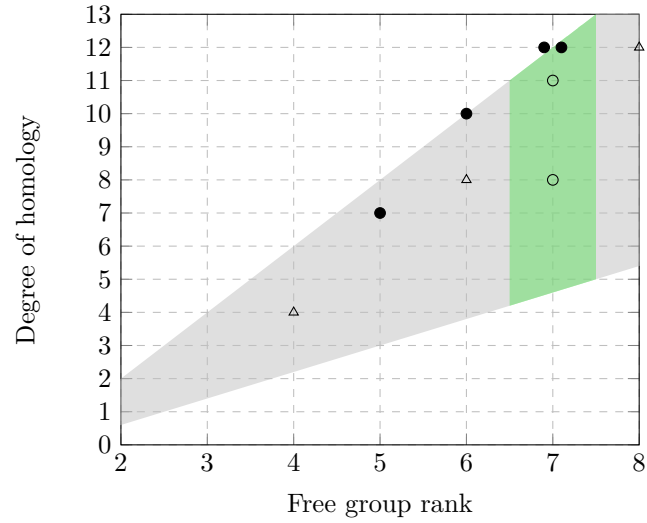
\begin{figure}
\begin{center}
\begin{tikzpicture}

\begin{axis}[
    xlabel={Free group rank},
    ylabel={Degree of homology},
    xmin=2, xmax=8,
    ymin=0, ymax=13,
	xtick distance=1,
	ytick distance=1,
    legend pos=north west,
    ymajorgrids=true,
    xmajorgrids=true,
    grid style=dashed,
	grid=both,
]

\addplot[
	only marks,
	black,
    color=black,
	fill=black,
	mark=triangle,
    ]
    coordinates {
    (4,4)
    (6,8)
    (8,12)
    };

\addplot[
	only marks,
	black,
    color=black,
	fill=black,
	mark=o,
    ]
    coordinates {
	(7, 8)
	(7, 11)
    };

\addplot[
	only marks,
	black,
    color=black,
	fill=black,
	mark=*,
    ]
    coordinates {
	(6.9, 12)
	(7.1, 12)

    };
\addplot[
	only marks,
	black,
    color=black,
	fill=black,
	mark=*,
    ]
    coordinates {
	(5,7)

    };

\addplot[
	only marks,
	black,
    color=black,
	fill=black,
	mark=*,
    ]
    coordinates {
	(6,10)

    };

\addplot[
	name path=upknown,
	domain=6.5:7.5,
	opacity=0,
]
	{2*x - 2};
\addplot[
	name path=upunknown,
	domain=0:12,
	opacity=0,
]
	{2*x - 2};
	
\addplot[
	name path=downknown,
	domain=6.5:7.5,
	opacity=0,
]
	{4*x/5 - 1};
\addplot[
	name path=downunknown,
	domain=0:12,
	opacity=0,
]
	{4*x/5 - 1};

\addplot [green, opacity=0.5] fill between [of=upknown and downknown];
\addplot [lightgray, opacity=0.5] fill between [of=upunknown and downunknown];
    
\end{axis}
\end{tikzpicture}
\end{center}

\caption{
\label{fig:results2}
The known rational homology of $\Aut(F_n)$. The shaded part is the potentially non-trivial range,
the green part contains the new results. The marks are filled if they represent a class
in the kernel of the map to the homology of $\Out(F_n)$. 
}
\end{figure}

\subsection*{On the vanishing of the Eisenstein classes}
There is a notable discrepancy of our tables with a previous statement in the literature. Namely, Conant-Kassabov-Vogtmann \cite{ConantKassabovVogtmann} defined Eisenstein classes $e_{4k+3}\in H_{4k+3}(\Aut(F_{2k+3}))$ by assembling classes in $H_{2k+3}(\Gamma_{2,2k+2})$ and $H_{2k}(\Gamma_{1,2k+1})$. They claimed that $e_{7}$ and $e_{11}$ are nonzero.
We claim in contrast that both are zero by the following arguments.
For $e_7\in H_7(\Aut(F_5))$ we can even show the stronger statement that $H_7(\Aut(F_5))$ does not contain any non-trivial assembled classes.
Namely, the assembled classes are those in the joint images of the assembly maps 
\begin{align*}
H_{7}(\Gamma_{4,3}) &\to H_7(\Aut(F_5)) & \\
H_{p}(\Aut(F_{g})) \otimes H_{7-p}(\Gamma_{5-g,2}) &\to H_7(\Aut(F_5)) &\text{for $0\leq p\leq 7$ and $1\leq g\leq 4$.}
\end{align*}
Consulting the tables of \cite{BrunWillwacher}, the domains of all these morphisms are zero. In particular, $e_7=0$. Furthermore, the argument also disproves \cite[Conjecture 10.1]{ConantHatcherKassabovVogtmann} of Conant--Hatcher--Kassabov--Vogtmann that the homology of the groups $\Gamma_{n,h}$ below the vcd is generated by assembled classes.

For the Eisenstein class $e_{11}\in H_{11}(\Aut(F_7))$ we first recall that $e_{11}$ must be in the kernel of the map $H_{11}(\Aut(F_7))\to H_{11}(\Out(F_7))$, arising from the projection $\Aut(F_7)\to \Out(F_7)$.\footnote{To see this, note that $e_{11}$ is assembled from two classes in the vcd. But the maps $\Gamma_{n,h}\to \Gamma_{n,h-1}$ send all vcd classes to zero by degree reasons.} Since the latter projection induces a split morphism in rational homology \cite[Theorem 2.4]{ConantHatcherKassabovVogtmann}, the map $H_{11}(\Aut(F_7))\to H_{11}(\Out(F_7))$ must be surjective. By Table \ref{tbl:results} the dimensions of domain and target are both 1, so the map must be injective as well. Hence $e_{11}=0$.

\subsection*{Acknowledgements}
We thank Karen Vogtmann and Martin Kassabov for helpful discussions.
The first two authors have been supported by the MSCA CaLIGOLA (grant number 101086123), funded by the European Commission, Charles University Research Center program No. UNCE/24/SCI/022, the SVV-2025-260837 project, and the GA UK project No. 253271. The third author has been supported by the NCCR Swissmap, funded by the Swiss National Science Foundation.

\subsection*{Declaration of AI use}
We have used AI tools (ChatGPT Pro) for proofreading and for general suggestions on the manuscript in the final stages of this project.
We also use some elements of code that were written by the AI.

\section{The forested graph complex}
\label{sec:FG}
We recall here briefly the definition of the hairy forested graph complex, see also \cite{BrunWillwacher} for a slightly more detailed exposition and proofs of some results we use. Note that we shall always work with the bridgeless variant.
In this article a \textit{hairy graph} is a connected 1-dimensional CW complex with $h$ univalent vertices labeled
by the elements of $\{1, \dots, h \}$, called \textit{hairs}, and the other vertices at least
3-valent.
For a hairy graph $G$, let $V(G)$ be its set of vertices and $E(G)$ its set of edges.
The \textit{cyclomatic rank} (or just \textit{rank}, or \emph{loop order}) of $G$ is the number
\[ |E(G)| - |V(G)| + 1. \]
Equivalently, this is the dimension of the first homology group of $G$ or of its cycle space.
We say that a hairy graph is \emph{2-edge connected} or \emph{bridgeless} if it remains connected after removing any edge between two non-hair vertices.
A tree is a connected graph with rank $0$, a forest is a disjoint union of trees.

We define the (bridgeless) \textit{forested graph complex} $\FG$ to be generated by the triples
$(G, F, \sigma)$, where $G$ is a bridgeless hairy graph, $F$ is a subforest of $G$ spanning all vertices
except the hairs (where a component of $F$ may consist of a single vertex) and $\sigma$ is the
\textit{orientation} - an ordering of edges of $F$, modulo the following equivalence relations:
\begin{itemize}
	\item For each permutation $\pi$ of $E(F)$, 
	\[ (G, F, \sigma) = \mathrm{sgn}(\pi) (G, F, \pi \circ \sigma) \]
	where $\pi \circ \sigma$ is the order $\sigma$ permuted by $\pi$.
	\item For each graph isomorphism $\phi: G \to H$ preserving the hair labels,
	\[ (G, F, \sigma) = (H, \phi(F), \phi(\sigma)) \]
	where $\phi(\sigma)$ is the order on $E(\phi(F))$ given by
	$e < e' \Leftrightarrow \phi^{-1}(e) < \phi^{-1}(e')$.
\end{itemize}
In particular, if $G$ has an automorphism permuting the order of edges of $F$ with the minus
sign, then $(G, F, \sigma) = 0$.

$\FG$ is graded by the number of edges in the chosen subforest. It carries
two differentials (for the order $\sigma$ being $e_1 < \dots < e_n$):
\begin{itemize}
	\item the \textit{unmarking differential}
	\begin{equation}\label{equ:du def}
    d_u((G, F, \sigma)) = \sum_{i = 1}^n (-1)^i(G, F \backslash e_i, \sigma \backslash e_i ) 
    \end{equation}
	where $F \backslash e_i$, resp. $\sigma \backslash e_i$ is $F$, resp. $\sigma$ with the
	edge $e_i$ removed.
	\item the \textit{contracting differential}
	\begin{equation}\label{equ:dc def}
    d_c((G, F, \sigma)) = \sum_{i = 1}^n (-1)^i(G_{e_i}, F_{e_i}, \sigma \backslash e_i ) 
    \end{equation}
	where $G_{e_i}$, resp. $F_{e_i}$ is $G$, resp. $F$ with the
	edge $e_i$ contracted. Since contracting edges can never decrease the edge-connectivity $G_{e_i}$ is again bridgeless.
\end{itemize}

The differentials satisfy $d_u^2 = d_ud_c + d_cd_u = d_c^2$ = 0. We will consider
$\FG$ with the primary differential $d_u + d_c$.
The complex $\FG$ has a natural splitting into subcomplexes
$\FG(n,h)$ generated by rank $n$ graphs with $h$ hairs. It is a pivotal result
of Conant, Kassabov and Vogtmann that they compute the rational group homology of the groups
\[
\Gamma_{n, h} = 
\begin{cases}
    \Out(F_n) & \text{if $h=0$} \\
    \Aut(F_n) \ltimes F_n^{h-1} & \text{if $h>0$}
\end{cases}.
\]

\begin{theorem}[{\cite[Theorem 11.1]{ConantKassabovVogtmann}, \cite{ConantVogtmann}}]
For $2n+h\geq 3$ we have
\[ H_*(\Gamma_{n, h}, \mathbb{Q}) \simeq H_*(\FG(n, h), \mathbb{Q}). \]
\end{theorem}

\subsection{The excess filtration}
For a hairy graph $G$ with $h$ hairs, define its \textit{excess} by
\[ e(G) = 2|E(G)| - 3|V(G)| + 2h= \sum_{\substack{v \in V(G)\\ \text{$v$ not hair}}} (\mathrm{deg}(v) - 3). \]
The excess measures how many edge contractions of a trivalent graph are necessary to obtain $G$;
the trivalent graphs are exactly the ones with excess $0$.

Let $\FG^{(e)}$ be the graded subspace spanned by graphs with excess $e$.
It is immediate that
\begin{align*}
d_u(\FG^{(e)}) &\subseteq \FG^{(e)} \\
d_c(\FG^{(e)}) &\subseteq \FG^{(e + 1)} \\
\end{align*}

The following proposition follows from the Koszulness of the cyclic commutative operad. It was first noted by Kontsevich \cite{Kontsevich} for the full forested complex, see also the observations of \cite{BrunWillwacher} to cover the 2-edge-connected subcomplex as well. The result was already used in \cite{Ohashi}, \cite{Bartholdi}.
\begin{proposition}\label{exc0}
The homology of $\FG(n,h)$ with respect to the $d_c$ differential is concentrated in
excess $0$ as long as $n\geq 1$.
\end{proposition}
\begin{corollary}\label{restrictexcess0}
The following inclusion of complexes is a quasi-isomorphism as long as $n\geq 1$:
\[  (\ker\ d_c \restriction_{\FG^{(0)}}, d_u )
\subset (\FG, d_u + d_c) \]
\end{corollary}
The kernel $\ker\ d_c\restriction_{\FG^{(0)}}\subset \FG^{(0)}$ with differential $d_u$ is the Lie graph complex. 



\subsection{Odd variant and homology with coefficients in the sign representation}
There is also an "odd" variant $\FGO$ of the forested graph complex that differs from $\FG$ only by sign conventions.
Concretely, elements of $\FGO$ are linear combinations of quadruples $(G,F,o,\lambda)$ with $G$ a hairy bridgeless graph as before, $F\subset G$ a sub-forest and $o\in \Lambda^{top}\Z E(F)$ and $\lambda\in \Lambda^{top} H_1(G,\Z)$ each a choice of one of the two generators.
We impose the relations, for $\phi:G\to H$ a graph isomorphism:
\begin{align*}
    (G,F,o,\lambda)&=- (G,F,-o,\lambda)=- (G,F,o,-\lambda)
\\
    (G,F,o,\lambda) &= (H,\phi(F), \phi(o),\phi(\lambda)).
\end{align*}
The differential (see \eqref{equ:du def} and \eqref{equ:dc def}) extends naturally and just leaves the extra datum $\lambda$ unchanged.
(Respectively, one applies the natural map $\Lambda^{top} H_1(G,\Z)\to \Lambda^{top} H_1(G/e,\Z)$.) The homological degree of the element $(G,F,o,\lambda)$ is still the number of edges in $F$.

Alternatively, the combined orientation datum $(o,\lambda)$ can also be encoded by fixing an ordering of the unmarked edges and vertices, and a direction of the edges, see \cite{BrunWillwacher} for details on this convention. By \cite{BrunWillwacher} we then have that 
\begin{align*}
    H_k(\FGO(n, h)) &= H_k(\Gamma_{n,h},\sgn), 
\end{align*}
where $\sgn$ is the representation of $\Gamma_{n,h}$ given by the composition 
\[
\Gamma_{n,h} \to \mathrm{GL}(n,\mathbb Z) \xrightarrow{\det} \{\pm 1\},
\]
and again the homology of $\FGO(n, h)$ with respect to the differential $d_c$ is concentrated in excess zero. 

\subsection{Three-edge connected subcomplex}
We say that a bridgeless hairy graph is \emph{3-edge-connected} if it remains connected after removing any set of two edges that are not adjacent to hairs.
The 3-edge-connected graphs span a subcomplex 
\[
(\FG^{3e}, d_c + d_u) \subset (\FG, d_c+d_u).
\]
Generally, the inclusion is not a quasi-isomorphism.
However, in the case without hairs the inclusion is close to a quasi-isomorphism, in the sense that the homology of the quotient
\[
\FG(n,0) / \FG^{3e}(n,0)
\]
is relatively small. We shall use this in our computation of the top degree homology of $\Out(F_9)$ below. 
For $n\leq 8$ one even has that the inclusion $\FG^{3e}(n,0) \to \FG(n,0)$ is a quasi-isomorphism, see Proposition \ref{prop:3ec} in the Appendix. 
We do not use the latter result for our computations, though.

\section{Methods and Results}
\label{sec:methods results}
Let $\FG(n,h)_p$ be the graded part of degree $p$. It is spanned by graphs of loop order $n$
with $h$ hairs and a chosen subforest with $p$ edges.
Fix $n$ and $h$ and let $\delta^p_u := d_u \restriction \FG^{(0)}(n,h)_p$ be the restriction of
$d_u$ to the excess $0$ part of $\FG(n, h)_p$, and analogously define $\delta^p_c$.
Let $\delta^p_{uc}$ be the sum $\delta^p_u + \delta^p_c$.
According to the Corollary \ref{restrictexcess0}, the $p$-th homology
$H_p(\FG(n, h))$ is isomorphic to the vector space
\[ \frac{\ker\ \delta^p_{uc}}
		{\delta^{p+1}_u (\ker\ \delta^{p+1}_c )} \]

By the rank-nullity theorem,
\[ \vdim \left(\ker\delta^p_{uc}\right) =
	\vdim \FG^{(0)}(n,h)_p  - \rank\ \delta^p_{uc} \]

Similarly,
\begin{align*}
\vdim\left( \delta^{p+1}_u (\ker \delta^{p+1}_c ) \right) &=
\vdim (\ker\delta^{p+1}_c)
- \vdim\left( 
 \ker\delta^{p+1}_c \cap
 \ker\delta^{p+1}_u\right) \\
&=
\vdim(\ker\delta^{p+1}_c)
- \vdim(
 \ker\delta^{p+1}_{uc}) \\
&= 
\vdim(\ker\delta^{p+1}_c)
- \vdim \FG^{(0)}(n,h)_{p+1}  + \rank\delta^{p+1}_{uc} \\
&= 
\rank\delta^{p+1}_{uc}
- \rank\delta^{p+1}_c
\end{align*}

In total,
\begin{align}
	\label{rankcomp}
	\vdim H_p(\FG(n,h)) = 
	\vdim \FG^{(0)}(n,h)_p  - \rank \delta^p_{uc} -	
\rank \delta^{p+1}_{uc}
+ \rank \delta^{p+1}_c
\end{align}

Moreover, thanks to the Proposition \ref{exc0}, the dimension of $\ker\ \delta^p_c$
(or from it the rank of $\delta^p_c$) can be computed using the Euler characteristic
as an alternating sum over different forest sizes and excesses
\[
	\vdim (\ker\delta^p_c) = \sum_{d + e = p} (-1)^e\vdim\FG^{(e)}(n,h)_d
\]

Thus the dimensions of the homology of $\FG(n,h)$ can be computed
using the knowledge of dimensions of the components for every excess and the ranks of the
matrices for the differential $\delta_{uc}$ in excess $0$, which is what we computed for loop orders $\leq 8$.

\subsection{The computer implementation}
We used the \texttt{nauty} library \cite{nauty} to generate all the graphs of the given
loop order of the specified excess. Since our graphs are allowed to have multiple edges
and tadpoles, we represented them as bipartite simple graphs, where the vertices of one
colour represent the edges: the degree 2 vertices are edges between distinct vertices
and the degree 1 vertices are tadpoles. The vertices of the other color must have degree
at least 3.\footnote{This degree condition apparently miscounts the degree contributions of the tadpoles. However, if the valence of a vertex is two, with one or both of the edges encoding a tadpole, then either the graph would consist of that single vertex and the loop order is $\leq 2$, or it has a bridge. Either way the graph can be ignored for the purposes of our computation.}

We then wrote Rust code that finds all subforests of the provided graphs,
up to isomorphism, computes the matrices for the $d_u + d_c$ differential and saves them.

On the generated matrices, we ran a simple partial sparse Gaussian elimination, using only operations
that strictly decrease the number of nonzeros, in the following steps:
\begin{enumerate}
	\item For each row with a single nonzero entry, we delete this row, along with the
	corresponding column, since this column cannot contribute to the kernel; we remember
	to add $1$ to the final rank. 
	
	\item For each row with two nonzero entries in columns $i$ and $j$, we replace the column
	$i$ with a linear combination of columns $i$ and $j$ making the entry in this row $0$. 
    We then delete this row and the column $j$. We remember to add $1$ to the rank.
	 

	\item For each pair of linearly dependent rows, we keep only the first one.
\end{enumerate}
Of course, rows and columns that are identically zero may also be deleted.
The above steps are iterated until no further such simplifications are possible. 
Each of the steps, while relatively cheap, provides a significant reduction of the matrix size.

\begin{example}
The top matrix (degree 13) of the 3-edge connected complex in the loop order 8 originally
has \texttt{54 923 731} rows, \texttt{40 184 866} after the step 1,
\texttt{26 788 165} after the step 2 and \texttt{9 525 418} after the step 3.
The number of columns is reduced from \texttt{6 737 887} to \texttt{4 607 207}
and then to \texttt{1 060 387}.
\end{example}

We then estimated the ranks of the matrices using the block Coppersmith-Wiedemann algorithm \cite{Coppersmith, Skipness}
over a small prime,
partially implemented on the GPU.
This algorithm only provides a lower bound of the rational rank, although with a high probability
of the bound being optimal.

Since in the formula (\ref{rankcomp}), the ranks of the $\delta_{uc}$ matrices appear with
a negative sign, if the computed homology happens to be $0$, it certifies that the ranks
cannot be higher and the lower bounds are the actual ranks. Therefore, the only situation
where we need to find further witnesses for the rank size is when the computation yields
nonzero homology in two consecutive degrees.\footnote{This argument has already been used in \cite{Bartholdi}.} For $\Out(F_8)$ and $\Aut(F_7)$, this only happens
in the two topmost degrees, meaning we need to verify only the rank of the topmost differential.

We do this by finding an explicit basis for its kernel. The block Wiedemann algorithm
provides vectors in the kernel modulo a small prime $p$. By explicitly checking linear
combinations of these vectors over the integers, we have managed to find the elements of
the rational kernel, thus bounding the rank of the matrix from above and verifying that
the computed rank is correct.

The computed dimensions, ranks and homology groups of $\Out(F_8)$ are summarized in Table
\ref{fig:out8figure}. The rank in the top degree was certified by finding 3 elements in
the kernel, which generate the top homology:
\begin{enumerate}
\item \texttt{3 681 059} nonzeros, coefficient range -34 to 34
\item \texttt{2 636 652} nonzeros, coefficient range -48 to 48
\item \texttt{3 805 340} nonzeros, coefficient range -60 to 60
\end{enumerate}
The support of these vectors is minimal among all vectors in the top kernel,
as was verified by trying all linear combinations modulo a small prime, which provides
a lower bound on the possible support over integers.

The computed dimension, ranks and homologies of $\Aut(F_7)$ and $\Out(F_8)$ can be found in the tables of Appendix \ref{app:tables}.
The appendix also contains the analogous data for homology of the groups $\Gamma_{n,h}$,
extending the range computed in \cite{BrunWillwacher}. Again, the ranks of the top
differentials were certified by finding explicit elements of their kernels.
The ranks of the other differentials are certified correct if they take part in the computation of a vanishing homology group as explained above.
The ranks of the remaining differentials are (likely correct) lower bounds.  Hence for some homology groups the computed values are (likely correct) upper bounds. 
Some of those may in fact be shown to be exact, because the split projection $\Gamma_{n,h}\to \Aut(F_n)$ provides a matching lower bound on the homology.
Two further entries, namely $H_8(\Gamma_{4,4})$ and $H_8(\Gamma_{5,3})$ can be shown to be correct, because a matching lower bound on the dimension is provided by the partial assemblies of the Morita class in $H_8(\Out(F_{6}))$.
Namely, these partial assemblies provide an irreducible $S_4$ representation of type $[2,2]$ in $H_8(\Gamma_{4,4})$, showing that the dimension is at least 2. They also provide a one-dimensional irreducible representation of type $[2]$ in $H_8(\Gamma_{5,2})$. Gluing a (genus 0) tripod at one hair produces three classes, which are distinguished by the three forgetful maps, and hence all distinct. This shows that the dimension of $H_8(\Gamma_{5,3})$ is at least 3, matching the upper bound. This in turn certifies correctness of the ranks of the adjacent matrices, so that the other entries for $\Gamma_{5,3}$ are also exact.
Nevertheless, some non-certified values remain.
We indicated those upper bounds by "$\leq$" in the tables of Appendix \ref{app:tables}.

Our source code and the full integral kernel vectors for the top degree homology can be found under the following address.

\medskip 

\begin{center}
\url{https://github.com/grego/forestgc}
\end{center}


\subsection{The special case of $\Out(F_9)$}
The top degree homology of $\Out(F_9)$ displayed in Figure \ref{fig:results} was computed slightly differently, because the straightforward computation of the rank of the top degree differential was prohibitively expensive.
Instead, we use the following simplification. We desire to compute the kernel of the top differential $\delta_{uc}^{15}$. 
Let $v\in \FG^{(0)}(9,0)$ be an element of excess zero and degree 15 in the kernel, i.e., $\delta_{uc}^{15}v=0$.
Then clearly, the corresponding class $\bar v$ in the quotient $\FG(9,0)/\FG^{3e}(9,0)$ by the three-edge connected subcomplex satisfies $\bar \delta_{uc}^{15} \bar v=0$, where $\bar \delta_{uc}^{15}$ is the induced differential on the quotient.
Hence we may follow the following algorithm to compute $\ker \delta_{uc}^{15}$:
\begin{enumerate}
\item Compute an integral generating set of the rational kernel $\ker \bar \delta_{uc}^{15}$.
Remember the support $S$ of that generating set, which are all (non-3-edge-connected) forested graphs appearing non-trivially in the kernel vectors.
It turns out that the kernel is only two-dimensional and its support is fairly small, relative to the dimension of the graph complex. 
\item Let $V$ be the span of the degree 15 excess 0 part of $\FG^{3e}(9,0)$ and the forested graphs in $S$. 
Let $\tilde \delta_{uc}^{15}$ be the restriction of $\delta_{uc}^{15}$ to $V$.
Then we compute 
\[
\ker \tilde \delta_{uc}^{15} = \ker \delta_{uc}^{15}.
\]
The matrix of $\tilde \delta_{uc}^{15}$ is significantly smaller than that of $\delta_{uc}^{15}$, making the above computation (barely) feasible. 
\end{enumerate}

The following table shows the matrix dimensions and rank results.
\begin{center}
  \small
  \setlength{\tabcolsep}{5pt}
  \begin{tabular}{@{}lrrrr@{}}
    \toprule
    Matrix & Rows & Columns & Rank & Nullity \\
    \midrule
    $\bar \delta_{uc}^{15}$
      & $1\,066\,092\,084$ & $102\,997\,476$
      & $102\,997\,474$ & $2$ \\
    $\tilde \delta_{uc}^{15}$
      & $3\,245\,973\,555$ & $351\,644\,414$
      & $351\,644\,390$ & $24$ \\
    \bottomrule
  \end{tabular}
\end{center}
In each case, the simple Gaussian elimination as described above was performed on the matrices. Then a lower bound for the matrix ranks was computed modulo a prime $p$, along with generators for the kernel modulo $p$, using Coppersmith-Wiedemann. The modular kernel vectors were then lifted to integral kernel vectors to obtain matching coranks, and thus certify that the ranks are correct rationally.

\appendix

\section{On the 3-edge connected subcomplex}
We consider again the subcomplex
\[
(\FG^{3e}, d_c + d_u) \subset (\FG, d_c+d_u).
\]
spanned by 3-edge-connected graphs. The goal of this appendix is to show the following result.

\begin{proposition}\label{prop:3ec} $ $
    The inclusion of the 0-hair and loop order $n$ part $\FG^{3e}(n,0)\subset \FG(n,0)$ is a quasi-isomorphism as long as $n\leq 8$.
\end{proposition}
The proof also indicates how to bound the homology of the quotient $$Q_n:=\FG(n,0)/\FG^{3e}(n,0)$$ for higher $n$, and demonstrates that it is generally small, though possibly non-zero.

Before the proof, let us recall some graph theoretical background. Let $\Gamma$ be a 2-edge-connected graph. 
On the set of vertices of $\Gamma$ define an equivalence relation such that vertices $u$ and $v$ are equivalent if they are connected by (at least) three edge-disjoint paths in $\Gamma$. The equivalence classes correspond to the 3-edge-connected components (3ECCs) of $\Gamma$. Contracting the full subgraphs of each equivalence class produces a graph $C(\Gamma)$, the cactus of $\Gamma$, with the property that each edge of $C(\Gamma)$ is contained in a unique simple cycle. The following example illustrates the procedure:
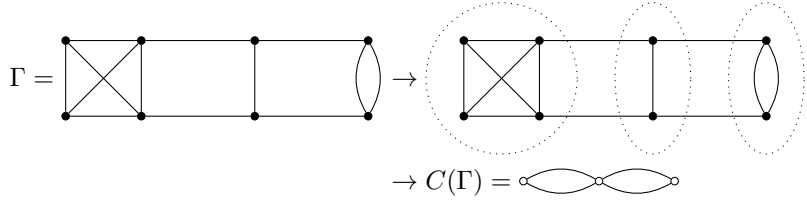
\begin{align*}
\Gamma = 
\begin{tikzpicture}
    \node[int] (v1) at (0,.5) {};
    \node[int] (v2) at (0,-.5) {};
    \node[int] (v3) at (1,.5) {};
    \node[int] (v4) at (1,-.5) {};
    \node[int] (w1) at (2.5,.5) {};
    \node[int] (w2) at (2.5,-.5) {};
    \node[int] (x1) at (4,.5) {};
    \node[int] (x2) at (4,-.5) {};
    \draw (v1) edge (v2) edge (v3) edge (v4) 
          (v2) edge (v3) edge (v4) 
          (v3) edge (v4) edge (w1) 
          (v4) edge (w2)
          (w1) edge (w2) edge (x1)
          (w2) edge (x2)
          (x1) edge[bend right] (x2) edge[bend left] (x2);
\end{tikzpicture}
&\to 
\begin{tikzpicture}
    \node[int] (v1) at (0,.5) {};
    \node[int] (v2) at (0,-.5) {};
    \node[int] (v3) at (1,.5) {};
    \node[int] (v4) at (1,-.5) {};
    \node[int] (w1) at (2.5,.5) {};
    \node[int] (w2) at (2.5,-.5) {};
    \node[int] (x1) at (4,.5) {};
    \node[int] (x2) at (4,-.5) {};
    \draw (v1) edge (v2) edge (v3) edge (v4) 
          (v2) edge (v3) edge (v4) 
          (v3) edge (v4) edge (w1) 
          (v4) edge (w2)
          (w1) edge (w2) edge (x1)
          (w2) edge (x2)
          (x1) edge[bend right] (x2) edge[bend left] (x2);
    \draw[dotted] (.5,0) circle (1);
    \draw[dotted] (2.5,0) ellipse (.5cm and 1cm);
    \draw[dotted] (4,0) ellipse (.5cm and 1cm);
\end{tikzpicture}
\\ & \to 
C(\Gamma)=
\begin{tikzpicture}
\node[ext] (U) at (0,0) {};
\node[ext] (V) at (1,0) {};
\node[ext] (W) at (2,0) {};
\draw (V) edge[bend right] (U) edge[bend left] (U)
    edge[bend right] (W) edge[bend left] (W);
\end{tikzpicture}
\end{align*}
We will call the vertices of $C(\Gamma)$ \emph{nodes} to distinguish them from the vertices of $\Gamma$.
Since each edge of $C(\Gamma)$ is contained in a unique simple cycle, the edges incident to each node are naturally paired, and in particular each node has even valence.
We say that the plain 3ECC corresponding to a node $V$ of $C(\Gamma)$ is the full subgraph of $\Gamma$ on the corresponding equivalence class of vertices.
The open 3ECC is the plain 3ECC together with the incident edges, which we consider as hairs. The closed 3ECC is obtained by pairing the hairs accordingly. For example for the middle node the corresponding plain, open and closed 3ECC are 
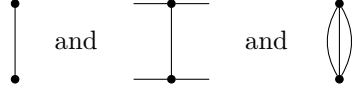
\[
\begin{tikzpicture}
    \node[int] (w1) at (2.5,.5) {};
    \node[int] (w2) at (2.5,-.5) {};
    \draw (w1) edge (w2);
\end{tikzpicture}
\quad\text{ and }\quad
\begin{tikzpicture}
    \node[int] (w1) at (2.5,.5) {};
    \node[int] (w2) at (2.5,-.5) {};
    \draw (w1) edge (w2) edge +(-.5,0) edge +(.5,0) 
    (w2) edge +(-.5,0) edge +(.5,0);
\end{tikzpicture}
\quad\text{ and }\quad
\begin{tikzpicture}
    \node[int] (w1) at (2.5,.5) {};
    \node[int] (w2) at (2.5,-.5) {};
    \draw (w1) edge (w2) edge[bend left] (w2) edge[bend right] (w2);
\end{tikzpicture}\, .
\]
In particular the plain and open 3ECCs might not be 3-edge connected (in fact, not even connected), while the closed 3ECC is always 3-edge connected. 
The graph $\Gamma$ may be recovered from $C(\Gamma)$ together with the 3ECCs associated to each component, by gluing the 3ECCs as is naturally prescribed by $C(\Gamma)$. Furthermore, each edge of $\Gamma$ corresponds uniquely to either an edge of $C(\Gamma)$ (then we call it a \emph{cactus edge}) or an edge of a plain 3ECC.

Also note that $\Gamma$ is 3-edge connected iff $C(\Gamma)$ consists of a single node.
Otherwise $C(\Gamma)$ must always have at least two nodes of valence two.

Next, we have the following compatibility of the cactus decomposition with edge contraction:

\begin{lemma}
    Let $\Gamma$ be a 2-edge-connected graph and fix some edge $e=(x,y)$ of $\Gamma$.
    Let $\Gamma/e$ be obtained by contracting $e$. 
    Then if $e$ is (i.e., corresponds to) an edge of a plain 3ECC of $\Gamma$ then $C(\Gamma/e)\cong C(\Gamma)$. Otherwise, $C(\Gamma/e)$ is obtained from $C(\Gamma)$ by contracting the edge (corresponding to) $e$, and removing a self-loop if one is created.  
\end{lemma}


Finally, we study an auxiliary complex $Y$ that is a subquotient of the 2-hair forested graph complex $\FG(-,2)$. We take from that complex the subcomplex $Y'\subset\FG(-,2)$ that consists of graphs that become 3-edge connected when we connect the two hairs. Then we set
\[
Y = Y' /I,
\]
where $I$ is the subcomplex spanned by all marked graphs $(\Gamma,T)$ such that the vertices adjacent to the hairs are not connected in the forest $T$.
In other words, $Y$ is spanned by graphs such that the two hairs are attached to the same connected component in $T$.
We will need the following:
\begin{lemma}\label{lem:Y}
    The $g$-loop part $Y^g\subset Y$ satisfies
    \[
    H(Y^0)=H(Y^1)=H(Y^2)=H(Y^3)=0.
    \]
\end{lemma}
\begin{proof}
    We have that $Y^0=0$ since there is no tree with at least one $\geq 3$-valent vertex and two leaves. $Y^1$ is spanned by two graphs which cancel under the differential.
    \begin{align*}
            \begin{tikzpicture}
            \node[int] (v) at (0,0) {};
            \node[int] (vv) at (.7,0) {};
            \node (e1) at (-.7,0) {$\scriptstyle 1$};
            \node (e2) at (1.4,0) {$\scriptstyle 2$};
            \draw (v) edge (vv) edge (e1) edge[bend left, dashed] (vv) 
                  (vv) edge (e2);
        \end{tikzpicture}
        \mapsto 
        \begin{tikzpicture}
            \node[int] (v) at (0,0) {};
            \node (e1) at (-.7,0) {$\scriptstyle 1$};
            \node (e2) at (.7,0) {$\scriptstyle 2$};
            \draw (v) edge[loop above, dashed] (v) edge (e1) edge (e2);
        \end{tikzpicture}
    \end{align*}
    The cases $g=2,3$ are similar, finite-dimensional, but more elaborate computations which we checked on the computer. 
\end{proof}

We also note that $Y^g$ is not generally acyclic. Again by numerical computation we find that $H_k(Y^4)$ is concentrated in the top degree $k=7$ and three-dimensional there, in the underlying $S_2$-representation $V_{2}\oplus 2V_{1^2}$.

\begin{proof}[Proof of Proposition \ref{prop:3ec}]
The statement that the inclusion $\FG^{3e}(n,0)\subset \FG(n,0)$ is a quasi-isomorphism for $n\leq 8$ is equivalent to showing that the quotient complex $Q_n$ generated by 2-edge- but not 3-edge-connected graphs is acyclic for $n\leq 8$.
To see that $Q_n$ is acyclic, we endow $Q_n$ with a filtration by declaring $\Gamma$ to be of filtration degree equal to the number of edges in the cactus $C(\Gamma)$.
It is then sufficient to show that the associated graded complex $\gr Q_n$ is acyclic.
We may identify
\[
\gr Q_n \cong  (Q_n, d_c' + d_u' + d_u^{cactus})
\]
where $d_c'$ and $d_u'$ contract (resp. unmark) edges in one of the 3ECCs, while $d_u^{cactus}$ unmarks a cactus edge. In other words, by passing to the associated graded we have removed from the differential the term that contracts marked cactus edges.

To go further, consider one open 3ECC $\gamma$ of a graph $\Gamma$ with marked sub-forest $T\subset \Gamma$. Let $\gamma'\subset \gamma$ be the corresponding plain 3ECC, which is just obtained by removing the hairs of $\gamma$. Then $T':=T\cap \gamma'$ is a marked subforest of the hairy graph $\gamma$. As such it partitions the set $H_\gamma$ of hairs of $\gamma$ into say $k_\gamma$ disjoint subsets, with two hairs $h,h'$ in the same subset if their adjacent vertices in $\gamma'$ are connected in $T'$. The differential can break connectivity in this sense, but never create it.
Hence we may filter $\gr Q_n$ by the number $K_\Gamma:=\sum_{\gamma} k_\gamma$, with the sum over all 3ECCs.
Again it suffices to check that the associated graded $\gr^K\gr Q_n$ is acyclic.
We have 
\[
\gr^K\gr Q_n \cong (Q_n, d_c'+d_u''+d_{u}^{cactus}),
\]
where now $d_u''$ unmarks edges in the 3ECCs, but only if this does not disconnect hairs in the 3ECCs in the sense above. In other words, the differential now leaves invariant both the cactus $C(\Gamma)$, as well as the partitions $P=\{P_U\}_U$ associated to each vertex $U$ of $C(\Gamma)$, with $P_U$ being a partition of the set of edges incident to $U$.
Furthermore, our complex splits:
\newcommand{\Ca}{\mathrm{Ca}}
\begin{equation}\label{equ:big split}
\bigoplus_n \gr^K\gr Q_n \cong \bigoplus_{C}\left(\bigoplus_P \Ca_{C,P} \otimes \bigotimes_{U\in VC} Y_{P_U, C, U}\right)_{\Aut(C)},
\end{equation}
where the outer sum is over isomorphism classes of cactus graphs with at least two vertices, and the inner sum over assignments of partitions $P=\{P_U\}_U$ of the sets of edges incident to each vertex of $C$.
The complex $\Ca_{C,P}$ is a complex of linear combinations of admissible markings of $C$, where an admissible marking is a subset of the edge set such that after splitting all vertices according to the partition the marked edges do not contain a cycle.
In the following we shall call the latter graph the $P$-exploded graph of $C$.
The complexes $Y_{P_U,C,U}$ associated to each vertex $U$ are spanned by the possible 3ECCs at that vertex, with a marked forest, so that the hair partition is $P_U$. The precise form of the $Y_{P_U,C,U}$ will be irrelevant for us, except noting that if $U$ has valence 2 in $C$ and $P_U$ is the trivial partition (i.e., $\{\{1,2\}\}$ for an appropriate numbering), then $Y_{P_U,C,U}=Y$ as defined before.

The motivation behind the filtration is that in \eqref{equ:big split} the parts of the complex corresponding to individual 3ECCs "decouple" from each other and the complex of markings on the cactus, so that we may compute the homology separately in each factor of the tensor product.

Now we study the factors individually:
First note that $H(\Ca_{C,P})=0$ if the exploded graph has univalent vertices, as marking an edge connected to a univalent vertex provides a homotopy.
Since $C$ must have at least 2 bivalent vertices, say $U$, $V$, we hence conclude that nontrivial homology can arise only if $P_U$ and $P_V$ are the trivial partitions.
But then $Y_{P_U,C,U}$ and $Y_{P_V,C,V}$ are both isomorphic to the auxiliary complex $Y$ studied above. By Lemma \ref{lem:Y} we hence know that $H(Y^{g-loop})=0$ for $g\leq 3$. Hence even the minimal $C$ with two vertices leads to a graph of loop order $2g+1\geq 9$, so that below that loop order $H(\gr^K\gr Q_n)=0$ as desired.
\end{proof}

The 3-edge-connected graph complex $\FG^{3e}(n,0)$ is significantly smaller than its 2-edge-connected counterpart $\FG(n,0)$. However, to be able to efficiently compute the homology, we also need to be able to restrict to the part of excess 0, as in Corollary \ref{restrictexcess0}.
We conjecture that the analogous result for $\FG^{3e}(n,0)$ is true as well:
\begin{conjecture}\label{conj:3 exc 0}
    The homology $H(\FG^{3e}(n,0), d_c)$ is concentrated in excess 0 for all $n$.
\end{conjecture}
We have verified the conjecture for $n\leq 7$ by numerical computation.
If it holds, then using the 3-edge-connected complex can significantly reduce computational load. 
The following table contains the dimensions and ranks for $\FG^{3e}(8,0)$, with the last two lines in the tables conditional on Conjecture \ref{conj:3 exc 0}.
{
\small
\begin{center}
\begin{tabular}{|c|c c c c|}
\hline
degree $p$ & 6 & 7 & 8 & 9  \\
\hline
dim $\FG^{3e, (0)}(8, 0)_p$ & 9 892 784 & 20 586 818 & 34 453 618 & 46 244 526  \\
rank $\delta^p_{uc}$ & 9 741 866 & 20 434 053 & 34 334 488 & 46 176 403 \\
rank $\delta^p_{c}$ & 9 623 925 & 20 283 135 & 34 181 723 & 46 057 273 \\
\hline
dim $H_p(\Out(F_8))$ & 0 & 0 & 0 & 0 \\
\hline
\end{tabular}\par
\begin{tabular}{|c|c c c c|}
\hline
degree $p$ &  10 & 11 & 12 & 13  \\
\hline
dim $\FG^{3e, (0)}(8, 0)_p$ & 49 017 127 & 39 540 052 & 22 230 026 & 6 737 887 \\
rank $\delta^p_{uc}$ & 48 989 400 & 39 532 301 & 22 228 848 & 6 737 884 \\
rank $\delta^p_{c}$ & 48 921 277 & 39 504 574 & 22 221 097 & 6 736 709 \\
\hline
dim $H_p(\Out(F_8))$ & 0 & 0 & 3 & 3 \\
\hline
\end{tabular}
\end{center}
}
Mind that we obtain the same result for the homology as in Table \ref{fig:out8figure}, providing further evidence in favor of Conjecture \ref{conj:3 exc 0}, at a greatly reduced computational cost.

\section{The girth filtration}
We say the \emph{girth} of a forested graph $(G, F)$ is the length of the smallest cycle in the graph obtained from $G$ by contracting all edges in $F$. We can immediately see that contracting a marked edge in a forested graph does not change its girth and unmarking increases its girth by at most 1. 
Denote by $\FG^{3e,\geq 2}(n,h)\subset \FG^{3e}(n,h)$ the subspace spanned by graphs of girth $\geq 2$ and let 
\[
\FG^{3e,1}(n,h) = \FG^{3e}(n,h) / \FG^{3e,\geq 2}(n,h)
\]
be the quotient. Then we have a short exact sequence of chain complexes
\[
0\to \FG^{3e,\geq 2}(n,h) \to \FG^{3e}(n,h) \to \FG^{3e,1}(n,h) \to 0,
\]
that induces a corresponding long exact sequence in homology

\begin{multline}\label{equ:girth les}
\cdots \to H_k(\FG^{3e,\geq 2}(n,h))
\to H_k(\FG^{3e}(n,h))
\\
\to H_k(\FG^{3e,1}(n,h))
\xrightarrow{d_{1,k}} H_{k-1}(\FG^{3e,\geq 2}(n,h)) \to \cdots,
\end{multline}

where we denote by $d_{1,k}$ the connecting homomorphism.
The above girth long exact sequence allows us to obtain a lower bound on the dimension of the homology of $\FG^{3e,1}(n,h)$ that is even exact if $d_{1,k}$ has full rank, at a greatly reduced computational cost.

\begin{proposition}\label{prop:girth bound}
For fixed $n,h$ let 
\begin{align*}
h_k^1 &:= \vdim H_k(\FG^{3e,1}(n,h)) &
h_k^{\geq 2} &:= \vdim H_k(\FG^{3e,\geq 2}(n,h)) \\
    m_k &:= \min\{ h_k^1, h_{k-1}^{\geq 2} \}. & & 
\end{align*}
Then 
\[
\vdim H_k(\FG^{3e}(n,h))
\geq h_k^1 + h_k^{\geq 2} - m_{k+1}-m_k,
\]
and equality holds if $d_{1,k}$ and $d_{1,k+1}$ each have full rank $m_k$ resp. $m_{k+1}$.
\end{proposition}
\begin{proof}
    Let $r_k$ be the rank of $d_{1,k}$. Then from the girth long exact sequence \eqref{equ:girth les} above we see that 
    \[
    \vdim H_k(\FG^{3e}(n,h)) = h_k^1 + h_k^{\geq 2} - r_k - r_{k+1},
    \]
    and the statement of the proposition easily follows.
\end{proof}

By explicit numerical computation we observed the following:

\begin{proposition}\label{prop:d1 full rank}
For $n\leq 7$, $h=0$ and any $k$ the connecting homomorphism $d_{1,k}$ has full rank. 
\end{proposition}

More precisely, since the contraction differential does not change the girth, it follows from our verification of Conjecture \ref{conj:3 exc 0} for $n\leq 7$ that the $d_c$-homology of $\FG^{3e,1}(n,0)$ and $\FG^{3e,\geq 2}(n,0)$ is concentrated in excess zero. We can hence compute the homology of these complexes by restricting to the excess zero part, like we did for $\FG$ before.
We then obtained the following data, which together with the known homology tables for $H_k(\FG^{3e}(n,0))$ imply that $d_{1,k}$ has full rank. We also list the $n=8$ case for reference, though we do not know whether Conjecture \ref{conj:3 exc 0} holds for $n=8$.

\begin{center}

$n=4$:

\begin{tabular}{c|ccc}
Degree & 3 & 4 & 5 \\
\hline
Girth = 1 & 0 & 1 & 0 \\
Girth $\geq$ 2 & 0 & 0 & 0 \\
\hline
$H_*(\FG^{3e}(4,0))$ & 0 & 1 & 0 \\
\end{tabular}

\medskip

$n=6$:

\begin{tabular}{c|cccccc}
Degree & 4 & 5 & 6 & 7 & 8 & 9 \\
\hline
Girth = 1 & 0 & 0 & 2 & 0 & 1 & 0 \\
Girth $\geq$ 2 & 0 & 2 & 0 & 0 & 0 & 0 \\
\hline
$H_*(\FG^{3e}(6,0))$ & 0 & 0 & 0 & 0 & 1 & 0 \\
\end{tabular}

\medskip

$n=7$:

\begin{tabular}{c|ccccccc}
Degree & 5 & 6 & 7 & 8 & 9 & 10 & 11 \\
\hline
Girth = 1 & 0 & 0 & 0 & 4 & 1 & 0 & 1 \\
Girth $\geq$ 2 & 0 & 0 & 4 & 2 & 0 & 0 & 0 \\
\hline
$H_*(\FG^{3e}(7,0))$ & 0 & 0 & 0 & 1 & 0 & 0 & 1 \\
\end{tabular}

\medskip

$n=8$:

\begin{tabular}{c|cccccccc}
Degree & 6 & 7 & 8 & 9 & 10 & 11 & 12 & 13 \\
\hline
Girth = 1 & 0 & 0 & 0 & 27 & 4 & 0 & 3 & 3 \\
Girth $\geq$ 2 & 0 & 0 & 27 & 4 & 0 & 0 & 0 & 0 \\
\hline
$H_*(\FG^{3e}(8,0))$ & 0 & 0 & 0 & 0 & 0 & 0 & 3 & 3  \\
\end{tabular}

\end{center}

Also note that the resulting homology is always concentrated in the girth 1 part, except for $H_8(\FG^{3e}(7,0))$, which is concentrated in girth $\geq 2$. 
Proposition \ref{prop:d1 full rank} leads us to raise the following conjecture:

\begin{conjecture}\label{conj:girth}
    The connecting homomorphism $d_{1,k}$ has full rank for any $n$ and $k$ and $h=0$.
\end{conjecture}

If Conjectures \ref{conj:3 exc 0} and \ref{conj:girth} hold, then they could help to substantially reduce the computational cost of computing $H_k(\Out(F_n))$ by Proposition \ref{prop:girth bound}.
To illustrate the computational savings, consider the case of $n=7$, where the conjectures are true. Here the dimension of the largest degree component in the (excess zero part of the) full forested graph complex is around 2.6 million. In comparison, the individual 3-edge-connected girth components are fairly manageable as the following tables show. 

\begin{center}
\small
\begin{tabular}{|c|c c c c c c c|}
\hline
degree $p$ & 5 & 6 & 7 & 8 & 9 & 10 & 11  \\
\hline
dim $\FG^{3e, 1,(0)}(7,0)_p$ & 53 102 & 173 389 & 395 591 & 624 127 & 655 197 & 435 819 & 151 001 \\
rank $\delta^p_{uc}$ & 51 363 & 170 937 & 393 108 & 622 519 & 654 677 & 435 748 & 151 000 \\
rank $\delta^p_{c}$ & 50 545 & 169 198 & 390 656 & 620 036 & 653 073 & 435 229 & 150 929 \\
\hline
dim $H_p(\FG^{3e, 1,(0)}(7,0))$ & 0 & 0 & 0 & 4 & 1 & 0 & 1 \\
\hline
\end{tabular}

\begin{tabular}{|c|c c c c c c c|}
\hline
degree $p$ & 5 & 6 & 7 & 8 & 9 & 10 & 11  \\
\hline
dim $\FG^{3e, \geq 2, (0)}(7,0)_p$ & 114 916 & 177 944 & 174 575 & 93 264 & 20 604 & 985 & 0 \\
rank $\delta^p_{uc}$ & 112 410 & 176 462 & 174 033 & 93 178 & 20 604 & 985 & 0 \\
rank $\delta^p_{c}$ & 109 866 & 173 956 & 172 551 & 92 640 & 20 520 & 985 & 0 \\
\hline
dim $H_p(\FG^{3e, \geq 2, (0)}(7,0))$ & 0 & 0 & 4 & 2 & 0 & 0 & 0 \\
\hline
\end{tabular}

\end{center}

Mind that in the case $n\geq 9$ the inclusion $\FG^{3e}(n,0)\to \FG(n,0)$ is generally not a quasi-isomorphism. The quotient has small homology, which however also needs to be studied to derive statements on $H_k(\Out(F_n))$.

\section{Tables of numerical results}
\label{app:tables}

\begin{table}[H]
\small
\begin{center}
\begin{tabular}{|c|c c c c|}
\hline
degree $p$ & 6 & 7 & 8 & 9  \\
\hline
dim $\FG^{(0)}(8,0)_p$ & 48 898 745 & 92 090 388 & 139 340 908 & 168 883 559  \\
rank $\delta^p_{uc}$ & 48 180 639 & 91 442 938 & 138 890 887 & 168 652 756 \\
rank $\delta^p_{c}$ & 47 557 019 & 90 724 832 & 138 243 437 & 168 202 735 \\
\hline
dim $H_p(\Out(F_8))$ & 0 & 0 & 0 & 0 \\
\hline
\end{tabular}\par
\begin{tabular}{|c|c c c c|}
\hline
degree $p$ &  10 & 11 & 12 & 13  \\
\hline
dim $\FG^{(0)}(8,0)_p$ & 161 342 472 & 117 024 310 & 59 084 396 & 16 157 290 \\
rank $\delta^p_{uc}$ & 161 258 633 & 117 003 755 & 59 081 662 & 16 157 287 \\
rank $\delta^p_{c}$ & 161 027 830 & 116 919 916 & 59 061 107 & 16 154 556 \\
\hline
dim $H_p(\Out(F_8))$ & 0 & 0 & 3 & 3 \\
\hline
\end{tabular}
\caption{Dimension of excess 0 parts, ranks of $\delta_{uc}$, ranks of $\delta_c$
and the resulting homology of $\Out(F_8)$.}
\label{fig:out8figure}
\end{center}
\end{table}

\begin{table}[H]
\begin{center}
\small
\begin{tabular}{|c|c c c c|}
\hline
degree $p$ & 5 & 6 & 7 & 8  \\
\hline
dim $\FG^{(0)}(7,1)_p$ & 20 582 184 & 42 630 875 & 69 290 559 & 88 620 629 \\
rank $\delta^p_{uc}$ & 19 990 698 & 42 045 098 & 68 850 176 & 88 375 972 \\
rank $\delta^p_{c}$ & 19 542 757 & 41 453 612 & 68 264 399 & 87 935 589 \\
\hline
dim $H_p(\Aut(F_7))$ & 0 & 0 & 0 & 1 \\
\hline
\end{tabular}\par
\begin{tabular}{|c|c c c c|}
\hline
degree $p$ & 9 & 10 & 11 & 12  \\
\hline
dim $\FG^{(0)}(7,1)_p$ & 88 097 414 & 65 712 647 & 33 785 982 & 9 339 756 \\
rank $\delta^p_{uc}$ & 88 002 555 & 65 688 907 & 33 782 845 & 9 339 754 \\
rank $\delta^p_{c}$ & 87 757 899 & 65 594 048 & 33 759 105 & 9 336 618 \\
\hline
dim $H_p(\Aut(F_7))$ & 0 & 0 & 1 & 2 \\
\hline
\end{tabular}
\end{center}
\caption{\label{fig:aut7figure}
Dimension of excess 0 parts, ranks of $\delta_{uc}$, ranks of $\delta_c$
and the resulting homology of $\Aut(F_7)$}
\end{table}

\begin{table}[H]
\small
\begin{center}\begin{tabular}{|c|c c c c|}
\hline
degree $p$ & 0 & 1 & 2 & 3  \\
\hline
dim $\FG^{(0)}(6,2)_p$ & 4 999 & 72 326 & 510 036 & 2 287 435  \\
rank $\delta^p_{uc}$ & 0 & 39 404 & 396 151 & 2 029 624 \\
rank $\delta^p_{c}$ & 0 & 34 406 & 363 229 & 1 915 739 \\
\hline
dim $H_p(\Gamma_{6, 2})$ & 1 & 0 & 0 & 0 \\
\hline
\end{tabular}\par
\begin{tabular}{|c|c c c c|}
\hline
degree $p$ & 4 & 5 & 6 & 7  \\
\hline
dim $\FG^{(0)}(6,2)_p$ & 7 218 916 & 16 869 131 & 29 984 986 & 40 954 172  \\
rank $\delta^p_{uc}$ & 6 814 085 & 16 413 864 & 29 610 189 & 40 729 719 \\
rank $\delta^p_{c}$ & 6 556 274 & 16 009 033 & 29 154 922 & 40 354 922 \\
\hline
dim $H_p(\Gamma_{6, 2})$ & 0 & 0 & 0 & 0 \\
\hline
\end{tabular}\par
\begin{tabular}{|c|c c c c|}
\hline
degree $p$ &  8 & 9 & 10 & 11  \\
\hline
dim $\FG^{(0)}(6,2)_p$ & 42 687 751 & 32 912 781 & 17 300 428 & 4 853 849 \\
rank $\delta^p_{uc}$ & 42 594 705 & 32 888 481 & 17 297 194 & 4 853 842 \\
rank $\delta^p_{c}$ & 42 370 252 & 32 795 437 & 17 272 894 & 4 850 613 \\
\hline
dim $H_p(\Gamma_{6, 2})$ & 2 & 0 & 5 & 7 \\
\hline
\end{tabular}
\caption{\label{fig:gamma62figure} Dimension of excess 0 parts, ranks of $\delta_{uc}$, ranks of $\delta_c$
and the resulting homology of $\Gamma_{6, 2}$}
\end{center}
\end{table}

\begin{table}[H]
\small
\begin{center}\begin{tabular}{|c|c c c c c|}
\hline
degree $p$ & 0 & 1 & 2 & 3 & 4  \\
\hline
dim $\FG^{(0)}(5,3)_p$ & 5 905 & 78 996 & 500 586 & 1 978 189 & 5 413 269 \\
rank $\delta^p_{uc}$ & 0 & 42 840 &	389 282 & 1 761 222 & 5 126 559 \\
rank $\delta^p_{c}$ & 0 & 36 936 & 353 126 & 1 649 918 & 4 909 592 \\
\hline
dim $H_p(\Gamma_{5, 3})$ & 1 & 0 & 0 & 0 & 0 \\
\hline
\end{tabular}\par
\begin{tabular}{|c|c c c|}
\hline
degree $p$ & 5 & 6 & 7  \\
\hline
dim $\FG^{(0)}(5,3)_p$ & 10 782 996 & 15 986 293 &	17 655 003  \\
rank $\delta^p_{uc}$ & 10 518 320 & 15 814 345 & 17 579 162 \\
rank $\delta^p_{c}$ & 10 231 610 & 15 549 669 &	17 407 214 \\
\hline
dim $H_p(\Gamma_{5, 3})$ & 0 & 0 & $6$ \\
\hline
\end{tabular}\par
\begin{tabular}{|c|c c c |}
\hline
degree $p$ & 8 & 9 & 10   \\
\hline
dim $\FG^{(0)}(5,3)_p$ & 14 160 990 & 7 643 546 & 2 187 207 \\
rank $\delta^p_{uc}$ & 14 140 410 & 7 640 782 & 2 187 201 \\
rank $\delta^p_{c}$ & 14 064 575 & 7 620 205 & 2 184 441 \\
\hline
dim $H_p(\Gamma_{5, 3})$ & 3 & 4 & 6 \\
\hline
\end{tabular}
\caption{\label{fig:gamma53figure} Dimension of excess 0 parts, ranks of $\delta_{uc}$, ranks of $\delta_c$
and the resulting homology of $\Gamma_{5, 3}$}
\end{center}
\end{table}

\begin{table}[H]
\small
\begin{center}\begin{tabular}{|c|c c c c|}
\hline
degree $p$ & 0 & 1 & 2 & 3  \\
\hline
dim $\FG^{(0)}(4,4)_p$ & 5 700 & 68 445 & 381 282 & 1 299 714 \\
rank $\delta^p_{uc}$ & 0 & 37 132 &	297 925 & 1 164 181 \\
rank $\delta^p_{c}$ & 0 & 31 433 & 266 612 & 1 080 824 \\
\hline
dim $H_p(\Gamma_{4, 4})$ & 1 & 0 & 0 & 0 \\
\hline
\end{tabular}\par
\begin{tabular}{|c|c c c|}
\hline
degree $p$ & 4 & 5 & 6  \\
\hline
dim $\FG^{(0)}(4,4)_p$ & 3 008 241 & 4 950 360 & 5 872 416 \\
rank $\delta^p_{uc}$ & 2 863 966 & 4 847 565 & 5 824 366 \\
rank $\delta^p_{c}$ & 2 728 433 & 4 703 291 & 5 721 571 \\
\hline
dim $H_p(\Gamma_{4, 4})$ & 1 & 0 & $\leq 10$ \\
\hline
\end{tabular}\par
\begin{tabular}{|c|c c c|}
\hline
degree $p$ & 7 &  8 & 9    \\
\hline
dim $\FG^{(0)}(4,4)_p$ & 4 939 206 & 2 751 159 & 807 993 \\
rank $\delta^p_{uc}$ & 4 925 888 & 2 749 409 & 807 984 \\
rank $\delta^p_{c}$ & 4 877 848 & 2 736 092 & 806 236 \\
\hline
dim $H_p(\Gamma_{4, 4})$ & $\leq 1$ & 2 &	9  \\
\hline
\end{tabular}
\caption{\label{fig:gamma44figure} Dimension of excess 0 parts, ranks of $\delta_{uc}$, ranks of $\delta_c$
and the resulting homology of $\Gamma_{4, 4}$}
\end{center}
\end{table}


\begin{table}[H]
\small
\begin{center}\begin{tabular}{|c|c c c c c|}
\hline
degree $p$ & 0 & 1 & 2 & 3 & 4  \\
\hline
dim $\FG^{(0)}(4,5)_p$ & 68 445 & 904 770 & 5 563 365 & 21 049 125 & 54 578 895 \\
rank $\delta^p_{uc}$ & 0 & 468 044 & 4 248 301 & 18 634 464 & 51 630 003 \\
rank $\delta^p_{c}$ & 0 & 399 600 & 3 811 575 & 17 319 400 & 49 215 342 \\
\hline
dim $H_p(\Gamma_{4, 5})$ & 1 & 0 & 0 & 0 & 1 \\
\hline
\end{tabular}\par
\begin{tabular}{|c|c c c|}
\hline
degree $p$ & 5 & 6 & 7  \\
\hline
dim $\FG^{(0)}(4,5)_p$ & 102 103 365 & 140 977 905 & 143 839 620 \\
rank $\delta^p_{uc}$ & 99 632 486 & 139 562 103 & 143 305 502 \\
rank $\delta^p_{c}$ & 96 683 595 & 137 091 224 & 141 889 720 \\
\hline
dim $H_p(\Gamma_{4, 5})$ & 0 & $\leq 20$ & $\leq 5$ \\
\hline
\end{tabular}\par
\begin{tabular}{|c|c c c |}
\hline
degree $p$ & 8 & 9 & 10   \\
\hline
dim $\FG^{(0)}(4,5)_p$ & 105 830 745 & 52 154 280 & 13 640 925 \\
rank $\delta^p_{uc}$ & 105 709 035 & 52 141 322 & 13 640 895 \\
rank $\delta^p_{c}$ & 105 174 922 & 52 019 623 & 13 627 998 \\
\hline
dim $H_p(\Gamma_{4, 5})$ & $\leq 11$ & $\leq 61$ & 30 \\
\hline
\end{tabular}
\caption{\label{fig:gamma45figure} Dimension of excess 0 parts, ranks of $\delta_{uc}$, ranks of $\delta_c$
and the resulting homology of $\Gamma_{4, 5}$}
\end{center}
\end{table}

\begin{table}[H]
\small
\begin{center}\begin{tabular}{|c|c c c c|}
\hline
degree $p$ & 0 & 1 & 2 & 3  \\
\hline
dim $\FG^{(0)}(3,6)_p$ & 42 885 & 496 440 & 2 634 570 & 8 460 600 \\
rank $\delta^p_{uc}$ & 0 & 261 539 & 2 046 301 & 7 596 279 \\
rank $\delta^p_{c}$ & 0 & 218 655 & 1 811 400 & 7 008 010 \\
\hline
dim $H_p(\Gamma_{3, 6})$ & 1 & 0 & 0 & 0 \\
\hline
\end{tabular}\par
\begin{tabular}{|c|c c c|}
\hline
degree $p$ & 4 & 5 & 6 \\
\hline
dim $\FG^{(0)}(3,6)_p$ & 18 254 745 & 27 723 420 & 30 066 990 \\
rank $\delta^p_{uc}$ & 17 449 591 & 27 241 257 & 29 887 957 \\
rank $\delta^p_{c}$ & 16 585 270 & 26 436 118 & 29 405 794 \\
\hline
dim $H_p(\Gamma_{3, 6})$ & 15 & 0 & 126 \\
\hline
\end{tabular}\par
\begin{tabular}{|c|c c c|}
\hline
degree $p$ & 7 & 8 & 9    \\
\hline
dim $\FG^{(0)}(3,6)_p$ & 22 941 540 & 11 556 990 & 3 100 500 \\
rank $\delta^p_{uc}$ & 22 903 842 & 11 553 486 & 3 100 491 \\
rank $\delta^p_{c}$ & 22 724 935 & 11 515 788 & 3 097 056 \\
\hline
dim $H_p(\Gamma_{3, 6})$ & 0 & 69 &	9  \\
\hline
\end{tabular}
\caption{\label{fig:gamma36figure} Dimension of excess 0 parts, ranks of $\delta_{uc}$, ranks of $\delta_c$
and the resulting homology of $\Gamma_{3, 6}$.}
\end{center}
\end{table}

\begin{table}[H]
\small
\begin{center}
\begin{tabular}{|c|c c c c|}
\hline
degree $p$ & 6 & 7 & 8 & 9  \\
\hline
dim $\FGO^{(0)}(8,0)_p$ &
18 761 597 & 40 392 604 & 69 210 756 & 94 334 907 \\
rank $\delta^p_{uc}$ &
17 626 054 &	39 154 077	& 68 221 364 &	93 762 763 \\
rank $\delta^p_{c}$ &
16 863 890 &	38 018 534	& 66 982 838 &	92 773 371 \\
\hline
dim $H_p(\Out(F_8), \mathrm{sgn})$ & 0 & 1 & 0 & 0 \\
\hline
\end{tabular}\par
\begin{tabular}{|c|c c c c|}
\hline
degree $p$ &  10 & 11 & 12 & 13  \\
\hline
dim $\FGO^{(0)}(8,0)_p$ &
100 865 552 &	81 530 997 &	45 624 853 &	13 694 427 \\
rank $\delta^p_{uc}$ &
100 634 968 &	81 472 059 &	45 617 476 &	13 694 420 \\
rank $\delta^p_{c}$ &
100 062 824 &	81 241 475 &	45 558 538 &	13 687 043 \\
\hline
dim $H_p(\Out(F_8), \mathrm{sgn})$ & 0 & 0 & 0 & 7 \\
\hline
\end{tabular}
\caption{Dimension of excess 0 parts, ranks of $\delta_{uc}$, ranks of $\delta_c$
and the resulting homology of $\Out(F_8)$ with sign representation coefficients.}
\label{fig:out8oddfigure}
\end{center}
\end{table}

\begin{table}[H]
\small
\begin{center}
\begin{tabular}{|c|c c c c|}
\hline
degree $p$ & 5 & 6 & 7 & 8  \\
\hline
dim $\FGO^{(0)}(7,1)_p$ &
8 307 838 &	19 655 396 &	36 160 049 &	51 895 594 \\
rank $\delta^p_{uc}$ &
7 480 666 &	18 640 482 &	35 272 592 &	51 349 922 \\
rank $\delta^p_{c}$ &
7 001 223 &	17 813 310 &	34 257 678 &	50 462 466 \\
\hline
dim $H_p(\Aut(F_7), \mathrm{sgn})$ & 0 & 0 & 1 & 0 \\
\hline
\end{tabular}\par
\begin{tabular}{|c|c c c c|}
\hline
degree $p$ & 9 & 10 & 11 & 12  \\
\hline
dim $\FGO^{(0)}(7,1)_p$ &
57 474 139 &	47 477 154 &	26 860 920 &	8 085 290 \\
rank $\delta^p_{uc}$ &
57 245 842 &	47 417 194 &	26 853 189 &	8 085 283 \\
rank $\delta^p_{c}$ &
56 700 170 &	47 188 897 &	26 793 229 &	8 077 555 \\
\hline
dim $H_p(\Aut(F_7), \mathrm{sgn})$ & 0 & 0 & 3 & 7 \\
\hline
\end{tabular}
\caption{Dimension of excess 0 parts, ranks of $\delta_{uc}$, ranks of $\delta_c$
and the resulting homology of $\Aut(F_7)$ with sign representation coefficients.}
\label{fig:aut7oddfigure}
\end{center}
\end{table}

\begin{table}[H]
\small
\begin{center}\begin{tabular}{|c|c c c c|}
\hline
degree $p$ & 0 & 1 & 2 & 3  \\
\hline
dim $\FGO^{(0)}(6,2)_p$ & 1 082 & 19 326 & 163 770 & 860 985  \\
rank $\delta^p_{uc}$ & 0 & 5 943 & 88 994 & 618 899  \\
rank $\delta^p_{c}$ & 0 & 4 861 & 75 611 & 544 123 \\
\hline
dim $H_p(\Gamma_{6, 2}, \mathrm{sgn})$ & 0 & 0 & 0 & 0 \\
\hline
\end{tabular}\par
\begin{tabular}{|c|c c c c|}
\hline
degree $p$ & 4 & 5 & 6 & 7  \\
\hline
dim $\FGO^{(0)}(6,2)_p$ & 3 129 647 & 8 321 229 & 16 666 447 & 25 421 008  \\
rank $\delta^p_{uc}$ & 2 627 185 & 7 617 608 & 15 988 017 & 24 973 080 \\
rank $\delta^p_{c}$ & 2 385 099 & 7 115 146 & 15 284 397 & 24 294 651 \\
\hline
dim $H_p(\Gamma_{6, 2}, \mathrm{sgn})$ & 0 & 1 & $\leq 1$ & $\leq 1$ \\
\hline
\end{tabular}\par
\begin{tabular}{|c|c c c c|}
\hline
degree $p$ &  8 & 9 & 10 & 11  \\
\hline
dim $\FGO^{(0)}(6,2)_p$ & 29 339 008 & 24 847 084 & 14 229 982 & 4 301 842 \\
rank $\delta^p_{uc}$ & 	29 143 171 & 24 794 732 & 14 223 204 & 4 301 828 \\
rank $\delta^p_{c}$ & 28 695 244 & 24 598 895 & 14 170 852 & 4 295 058 \\
\hline
dim $H_p(\Gamma_{6, 2}, \mathrm{sgn})$ & 0 & 0 & 8 & 14 \\
\hline
\end{tabular}
\caption{\label{fig:gamma62oddfigure} Dimension of excess 0 parts, ranks of $\delta_{uc}$, ranks of $\delta_c$
and the resulting homology of $\Gamma_{6, 2}$ with sign representation coefficients.}
\end{center}
\end{table}

\begin{table}[H]
\small
\begin{center}\begin{tabular}{|c|c c c c c|}
\hline
degree $p$ & 0 & 1 & 2 & 3 & 4 \\
\hline
dim $\FGO^{(0)}(5,3)_p$ & 1 781 & 28 230 & 208 425 & 943 980  & 2 920 191 \\
rank $\delta^p_{uc}$ & 0 & 9 014 & 118 461 & 705 299 & 2 525 050 \\
rank $\delta^p_{c}$ & 0 & 7 233 & 99 245 & 615 335 & 2 286 369 \\
\hline
dim $H_p(\Gamma_{5, 3}, \mathrm{sgn})$ & 0 & 0 & 0 & 0 & 0 \\
\hline
\end{tabular}\par
\begin{tabular}{|c|c c c|}
\hline
degree $p$ & 5 & 6 & 7  \\
\hline
dim $\FGO^{(0)}(5,3)_p$ & 6 502 332 & 10 670 694 & 12 923 025 \\
rank $\delta^p_{uc}$ & 6 077 021 & 10 369 592 & 12 785 642 \\
rank $\delta^p_{c}$ & 5 681 880 & 9 944 290 & 12 484 540 \\
\hline
dim $H_p(\Gamma_{5, 3}, \mathrm{sgn})$ & 9 & 0 & $\leq 1$ \\
\hline
\end{tabular}\par
\begin{tabular}{|c|c c c |}
\hline
degree $p$ &  8 & 9 & 10   \\
\hline
dim $\FGO^{(0)}(5,3)_p$ & 11 261 145 & 6 540 355 & 1 990 926 \\
rank $\delta^p_{uc}$ & 11 223 824 & 6 535 571 & 1 990 919 \\
rank $\delta^p_{c}$ & 11 086 442 & 6 498 253 & 1 986 144 \\
\hline
dim $H_p(\Gamma_{5, 3}, \mathrm{sgn})$ & $\leq 3$ & $\leq 9$ & 7 \\
\hline
\end{tabular}
\caption{\label{fig:gamma53oddfigure} Dimension of excess 0 parts, ranks of $\delta_{uc}$, ranks of $\delta_c$
and the resulting homology of $\Gamma_{5, 3}$ with sign representation coefficients.}
\end{center}
\end{table}

\begin{table}[H]
\small
\begin{center}\begin{tabular}{|c|c c c c|}
\hline
degree $p$ & 0 & 1 & 2 & 3  \\
\hline
dim $\FGO^{(0)}(4,4)_p$ & 2 403 & 32 970 & 207 273 & 788 019 \\
rank $\delta^p_{uc}$ & 0 & 11 363 & 126 154 & 621 566 \\
rank $\delta^p_{c}$ & 0 & 8 960 & 104 547 & 540 447 \\
\hline
dim $H_p(\Gamma_{4, 4}, \mathrm{sgn})$ & 0 & 0 & 0 & 1 \\
\hline
\end{tabular}\par
\begin{tabular}{|c|c c c|}
\hline
degree $p$ & 4 & 5 & 6  \\
\hline
dim $\FGO^{(0)}(4,4)_p$ & 2 011 959 & 3 614 718 & 4 634 844 \\
rank $\delta^p_{uc}$ & 1 807 265 & 3 458 962 & 4 561 849 \\
rank $\delta^p_{c}$ & 1 640 813 & 3 254 283 & 4 406 099 \\
\hline
dim $H_p(\Gamma_{4, 4}, \mathrm{sgn})$ & $\leq 15$ & $\leq 6$ & 0 \\
\hline
\end{tabular}\par
\begin{tabular}{|c|c c c|}
\hline
degree $p$ & 7 & 8 & 9 \\
\hline
dim $\FGO^{(0)}(4,4)_p$ & 4 171 686 & 2 461 443 & 757 449 \\
rank $\delta^p_{uc}$ & 4 151 789 & 2 459 001 & 757 430 \\
rank $\delta^p_{c}$ & 4 078 794 & 2 439 125 & 754 993 \\
\hline
dim $H_p(\Gamma_{4, 4}, \mathrm{sgn})$ & $\leq 21$ & $\leq 5$ & 19  \\
\hline
\end{tabular}
\caption{\label{fig:gamma44oddfigure} Dimension of excess 0 parts, ranks of $\delta_{uc}$, ranks of $\delta_c$
and the resulting homology of $\Gamma_{4, 4}$ with sign representation coefficients.}
\end{center}
\end{table}


\begin{table}[H]
\small
\begin{center}\begin{tabular}{|c|c c c c c|}
\hline
degree $p$ & 0 & 1 & 2 & 3 & 4 \\
\hline
dim $\FGO^{(0)}(4,5)_p$ & 32 970 & 483 450 & 3 269 925 & 13 498 665 & 37 888 560 \\
rank $\delta^p_{uc}$ & 0 & 173 940 & 2 032 620 & 10 731 315 & 34 058 243 \\
rank $\delta^p_{c}$ & 0 & 140 970 & 1 723 110 & 9 494 010 & 31 290 894 \\
\hline
dim $H_p(\Gamma_{4, 5}, \mathrm{sgn})$ & 0 & 0 & 0 & 1 & $\leq 35$ \\
\hline
\end{tabular}\par
\begin{tabular}{|c|c c c|}
\hline
degree $p$ & 5 & 6 & 7 \\
\hline
dim $\FGO^{(0)}(4,5)_p$ & 76 136 445 & 112 062 075 & 120 963 300 \\
rank $\delta^p_{uc}$ & 72 709 292 & 110 058 774 & 120 216 513 \\
rank $\delta^p_{c}$ & 68 879 010 & 106 631 631 & 118 213 212 \\
\hline
dim $H_p(\Gamma_{4, 5}, \mathrm{sgn})$ & $\leq 10$ & 0 & $\leq 45$ \\
\hline
\end{tabular}\par
\begin{tabular}{|c|c c c |}
\hline
degree $p$ & 8 & 9 & 10   \\
\hline
dim $\FGO^{(0)}(4,5)_p$ & 93 452 985 & 47 999 820 & 12 984 285 \\
rank $\delta^p_{uc}$ & 93 289 726 & 47 983 271 & 12 984 275 \\
rank $\delta^p_{c}$ & 92 542 984 & 47 820 037 & 12 967 821 \\
\hline
dim $H_p(\Gamma_{4, 5}, \mathrm{sgn})$ & $\leq 25$ & $\leq 95$ & 10 \\
\hline
\end{tabular}
\caption{\label{fig:gamma45oddfigure} Dimension of excess 0 parts, ranks of $\delta_{uc}$, ranks of $\delta_c$
and the resulting homology of $\Gamma_{4, 5}$ with sign representation coefficients.}
\end{center}
\end{table}

\begin{table}[H]
\small
\begin{center}\begin{tabular}{|c|c c c c|}
\hline
degree $p$ & 0 & 1 & 2 & 3  \\
\hline
dim $\FGO^{(0)}(3,6)_p$ & 27 765 & 345 240 & 1 955 430 & 6 658 440 \\
rank $\delta^p_{uc}$ & 0 & 140 400 & 1 335 405 & 5 647 585 \\
rank $\delta^p_{c}$ & 0 & 112 635 & 1 130 565 & 5 027 560 \\
\hline
dim $H_p(\Gamma_{3, 6}, \mathrm{sgn})$ & 0 & 0 & 0 & 41 \\
\hline
\end{tabular}\par
\begin{tabular}{|c|c c c|}
\hline
degree $p$ & 4 & 5 & 6  \\
\hline
dim $\FGO^{(0)}(3,6)_p$ & 15 132 645 & 24 048 180 & 27 113 190 \\
rank $\delta^p_{uc}$ & 14 152 791 & 23 464 187 & 26 901 970 \\
rank $\delta^p_{c}$ & 13 141 977 & 22 484 333 & 26 317 997 \\
\hline
dim $H_p(\Gamma_{3, 6}, \mathrm{sgn})$ & 0 & 20 & 0\\
\hline
\end{tabular}\par
\begin{tabular}{|c|c c c|}
\hline
degree $p$ & 7 &  8 & 9    \\
\hline
dim $\FGO^{(0)}(3,6)_p$ & 21 369 060 & 11 050 470 & 3 024 900 \\
rank $\delta^p_{uc}$ & 21 326 305 & 11 046 794 & 3 024 854 \\
rank $\delta^p_{c}$ & 21 115 085 & 11 004 165 & 3 021 183 \\
\hline
dim $H_p(\Gamma_{3, 6}, \mathrm{sgn})$ & $\leq 126$ & $\leq 5$ & 46\\
\hline
\end{tabular}
\caption{\label{fig:gamma36oddfigure} Dimension of excess 0 parts, ranks of $\delta_{uc}$, ranks of $\delta_c$
and the resulting homology of $\Gamma_{3, 6}$ with sign representation coefficients.}
\end{center}
\end{table}

\end{document}